\documentclass{article} 
\usepackage{arxiv}

\usepackage[english]{babel} 
\usepackage{amssymb}
\usepackage{amsmath}
\usepackage{amsthm}
\usepackage{mathtools}

\usepackage{abstract}

\newtheorem{lemma}{Lemma}
\newtheorem{thm}{Theorem}

\newtheorem{prop}{Proposition}
\newtheorem{rem}{Remark}

\DeclarePairedDelimiter{\norm}{\lVert}{\rVert}
\DeclarePairedDelimiter{\abs}{\vert}{\vert}

\title{Convergence rates of RLT and Lasserre-type hierarchies for the generalized moment problem over the simplex and the sphere}

\author{
	Felix Kirschner
	\thanks{Tilburg University,	f.c.kirschner@tilburguniversity.edu   } 
	\and 
	\textbf{Etienne de Klerk}
	\thanks{Tilburg University,	e.deklerk@tilburguniversity.edu
	\newline This work is supported by the European Union’s Framework Programme for Research and Innovation Horizon 2020 under the Marie Sk\l odowska-Curie grant agreement N. 813211 (POEMA).}
}

\begin{document}
	\maketitle
	
	\begin{abstract}
		We consider the generalized moment problem (GMP) over the simplex and the sphere. This is a rich setting and it contains NP-hard
		problems as special cases, like constructing optimal cubature schemes and rational optimization. 
		Using the Reformulation-Linearization Technique (RLT) and Lasserre-type hierarchies, relaxations of the problem are introduced and analyzed. 
		For our analysis we assume throughout the existence of a dual optimal solution as well as strong duality.
		For the GMP over the simplex we prove a convergence rate of $O(1/r)$ for a linear programming, RLT-type hierarchy, 
		where $r$ is the level of the hierarchy, using a quantitative version of P\'olya's Positivstellensatz. As an 
		extension of a recent result by Fang and Fawzi [\textit{Math. Program.}, 2020, https://doi.org/10.1007/s10107-020-01537-7] we 
		prove the Lasserre hierarchy of the GMP [\textit{Math. Program.}, Vol. 112, 65-92, 2008] over the sphere has a convergence rate of $O(1/r^2)$. 
		Moreover, we show the introduced linear RLT-relaxation is a generalization of a hierarchy for minimizing forms of degree $d$ over the
		simplex,  introduced by de Klerk, Laurent and Parrilo [\textit{J. Theoretical Computer Science}, Vol. 361, 210-225, 2006].
	\end{abstract}
	
	\section{Introduction}
	For a compact set $K \subset \mathbb{R}^n$ let $\mathcal{M}(K)$ denote the (infinite-dimensional) vector space of 
	signed finite Borel measures with support contained in $K$. Let $[m]=\{1, \dots, m\}$ for $m \in \mathbb{N}$. The generalized moment problem (GMP) is an optimization problem of the following form:
	\begin{equation}\label{primal}
	\begin{aligned}
	\text{val} := \inf_{\mu \in \mathcal{M}(K)_+} & \int_{K} f_0(\textbf{x}) \mathrm{d}\mu(\textbf{x}) \\
	\text{s.t. } &\int_{K} f_i(\textbf{x}) \mathrm{d}\mu(\textbf{x}) = b_i \; \forall i \in [m]  \\
	& \int_{K}  \mathrm{d}\mu(\textbf{x}) \le 1,
	\end{aligned}
	\end{equation}
	where
	$m \in \mathbb{N}, b_i \in \mathbb{R}$ for all $i \in [m]$, $\mathcal{M}(K)_+$ is the convex cone of positive finite Borel measures supported on $K$, and
	$f_0, f_1, \dots, f_{m}$ are integrable over $K$ with respect to all $\mu \in \mathcal{M}(K)_+$.
	We will always assume the GMP \eqref{primal} has a feasible solution, which implies that it has an optimal solution as well (see Theorem \ref{cor1}).
	
	The constraint $\int_{K}  \mathrm{d}\mu(\textbf{x}) \le 1$ essentially means that we know an upper  bound on the measure of $K$ for the optimal solution, since, in this case,
	we may scale the functions $f_i$ a priori to satisfy this condition.
	
	The GMP is a conic linear optimization problem whose duality theory is well understood, see e.g. \cite{shapiro}.
	A wide range of optimization problems can be modeled as an instance of the GMP.
	The list includes problems from optimization, probability, financial economics and optimal control to name only a few, see e.g. \cite{lasserre3}.
	For \emph{polynomial data}, i.e., $f_i \in \mathbb{R[\textbf{x}]}$ for all $i = 0,1, \dots, m$ and the set $K$
	being a basic closed semialgebraic set, Lasserre \cite{lasserreGMP} introduced a monotone nondecreasing hierarchy of semidefinite programming (SDP) 
	relaxations of (\ref{primal}). For a survey on SDP-based approximation hierarchies and their error analysis, see \cite{etmon2}.
	
	In this paper, we will consider the case where $K$ is the standard (probability) simplex
	\[ \Delta_{n-1} = \left \{ \textbf{x} \in \mathbb{R}_+^n : x_1+ \dots + x_n = 1 \right\}, \]
	where $\mathbb{R}^n_+$ is the nonnegative orthant, or the Euclidean sphere
	\[ \mathcal{S}^{n-1} = \left \{ \textbf{x} \in \mathbb{R}^n : \norm{\textbf{x}}_2^2 = x_1^2 + \dots + x_n^2 = 1 \right \}. \]
	Our main result is to establish a rate of convergence for the Lasserre hierarchy \cite{lasserreGMP} for the GPM on the sphere, and for a related,
	RLT-type linear programming hierarchy for the GPM on the simplex. This RLT hierarchy is in fact a generalisation of LP hierarchies for polynomial
	optimization on the simplex, as introduced by Bomze and De Klerk \cite{deklerk2}, and
	De Klerk, Laurent and Parrilo \cite{PTAS}.

	\textbf{Outline of the paper.} First we introduce some notation in section \ref{notation}. In section \ref{dualtheory} we 
	review the duality theory of the GMP. 
	A brief overview of possible applications of our setting is given in section \ref{applications}. 
	For $K$ the simplex we introduce a linear relaxation hierarchy in this setting in section \ref{chapterLP} and prove a convergence rate of $O(1/r)$.
	Section \ref{sec:Lasserre on sphere} contains the new convergence analysis of the Lasserre \cite{lasserre3} SDP hierarchies of the GPM on  the sphere.
	In Section \ref{sec:limiting behavior} we take a mathematical view of how the optimal measure is obtained in the limit as the level of the hierarchies approaches infinity.
	In section \ref{conclusion} we explain how our LP hierarchy is a generalization of an approximation hierarchy for the problem of minimizing a form of degree $d$ over the simplex introduced by De Klerk, Laurent and Parrilo \cite{PTAS} based on earlier results obtained by Bomze, De Klerk \cite{deklerk2}.
	
	\subsection{Notation}\label{notation}
	Let $\mathbb{N}= \{ 0, 1, 2, \dots \}$ denote the set of nonnegative integers, $\mathbb{N}_+ = \mathbb{N} \setminus \{ 0 \}$ and $\mathbb{N}_t^n$ the set of sequences $\alpha \in \mathbb{N}^n$ for which $\vert \alpha \vert = \sum_{i = 1}^{n} \alpha_i \le t$ for $t \in \mathbb{N}$. For $\alpha \in \mathbb{N}^n$, $\textbf{x}^\alpha$ denotes the monomial $x_1^{\alpha_1}\dots x_n^{\alpha_n}$ and its degree is $\vert \alpha \vert$. The ring of multivariate polynomials in $n$ variables $\textbf{x}= (x_1, \dots, x_n)$ is denoted by $\mathbb{R}[\textbf{x}]= \mathbb{R}[x_1, \dots, x_n]$  and $\mathbb{R}[\textbf{x}]_t$ is its subspace of polynomials of degree at most $t$. The (total) degree of a polynomial is the maximal degree of its appearing monomials.
	A monomial basis vector of order $t$ is given by
	\[ [\textbf{x}]_t = (1, x_1 , \dots , x_n , x_1^2 , x_1 x_2 , \dots , x_{n-1} x_n , x_n^2, \dots, x_1^t , \dots, x_n^t )^T.
	\]
	Any polynomial $p \in \mathbb{R}[\textbf{x}]$ can be written as $p = \sum_{\alpha \in
		\mathbb{N}^n} p_\alpha \textbf{x}^\alpha $, where only finitely many $p_\alpha$ are non-zero. A polynomial $p\in \mathbb{R}[\textbf{x}]$ is a sum of squares (sos) if $p = \sum_{j=1}^{k} (h_j)^2$ for $h_j \in \mathbb{R}[\textbf{x}]$ and $k \ge 1$. The set of sos polynomials is denoted by $\Sigma[\textbf{x}]$ and the set of sos polynomials of degree at most $t$ is denoted by $\Sigma[\textbf{x}]_t$.

	\subsection{Duality of the generalized problem of moments}\label{dualtheory}
	We shall briefly discuss the duality theory associated with the GMP. For this, let $\mathcal{C}(K)$ denote the space of bounded continuous 
	functions on $K$ endowed with the supremum norm $\norm{\cdot}_\infty$. For two vector spaces $E, F$ of arbitrary dimension, 
	a non-degenerate bilinear form $\langle \rangle : E \times F \rightarrow \mathbb{R}$ is called a duality of $E$ and $F$.
	The spaces $\mathcal{M}(K)$ and $\mathcal{C}(K)$ can be put in duality by 
	defining $\langle \rangle : \mathcal{C}(K) \times \mathcal{M}(K) \rightarrow \mathbb{R}$ as
	\begin{equation}\label{duality}
	\langle f, \mu \rangle = \int_K f(\textbf{x}) \mathrm{d}\mu (\textbf{x}).
	\end{equation}
	Let again $f_0, f_1, \dots, f_m$ be continuous functions on $K$ and $b_1, \dots, b_m \in \mathbb{R}$. The dual of  (\ref{primal}) is given by
	
	\begin{equation}\label{dual}
	\begin{aligned}
	\text{val}^\prime = \sup_{(y,t) \in \mathbb{R}^m\times \mathbb{R}_+} & \sum_{i = 1}^m y_i b_i - t \\
	\text{s.t. } & f_0(\textbf{x})-\sum_{i=1}^m y_i f_i(\textbf{x}) +t \ge 0 \quad \forall \; \textbf{x} \in K.
	\end{aligned}
	\end{equation}
	Note that the dual problem \eqref{dual} is always strictly feasible, due to the constraint $\int_{K} \mathrm{d}\mu\le 1$ in the primal GMP \eqref{primal}.
	
	Weak duality holds for this pair of problems, meaning $\text{val}^\prime \le \text{val}$. The difference $\text{val}-\text{val}^\prime $ is
	called \emph{duality gap}.
	In fact, the duality gap is always zero, as the next theorem shows. Note that a zero duality gap does \textit{not} imply the existence of a dual optimal solution.

	
	\begin{thm}(see, e.g. \cite[Theorem 1.3]{lasserre3})\label{cor1}
		Assume problem \eqref{primal} is feasible. Then it has an optimal solution (the $\inf$ is attained), and $\mathrm{val} = \mathrm{val}^\prime$.
	\end{thm}
	
	
	We continue by recalling a sufficient condition for a dual optimal solution to exist.
	
	\begin{thm}(see, e.g. \cite[Proposition 2.8]{shapiro})\label{suffcond}
		Suppose  problem \eqref{primal} is feasible. If
		\begin{equation}\label{condition1}
		b \in \text{\normalfont int} ((\langle f_1, \mu \rangle, \dots, \langle f_{m} , \mu \rangle ) :  \mu \in \mathcal{M}(K)_+ )
		\end{equation}
		then  the set of optimal solutions of (\ref{dual}) is nonempty and bounded.
	\end{thm}
	
	As discussed in Lasserre \cite{lasserreGMP}, it is customary in the literature to assume that condition (\ref{condition1}) holds, 
	but in practice it may be a non-trivial task to check whether it does. We do stress, however, that condition (\ref{condition1})
	does hold for the applications discussed in the next subsection.
	
	Another result worth mentioning is that if the GMP (\ref{primal}) has an optimal solution, it has on which is finite atomic.
	
	\begin{thm}(see, e.g. \cite[Theorem 3]{etmon2})\label{atomic}
		If the GMP (\ref{primal}) has an optimal solution, then it has one which is finite atomic with at most $m$ atoms, i.e., of th form $\mu^\ast = \sum_{\ell = 1}^m \omega_\ell \delta_{\normalfont{\textbf{x}}^{(\ell)}}$ where $\omega_\ell \ge 0, \normalfont{\textbf{x}}^{(\ell)} \in K$ and $\delta_{\normalfont{\textbf{x}}^{(\ell)}}$ denotes the Dirac measure supported at $\normalfont{\textbf{x}}^{(\ell)} (\ell \in [m])$.
	\end{thm}

	\subsection{Applications}\label{applications}
	
	\textbf{Polynomial and rational optimization.} Consider the problem of minimizing a rational function over $K$:
	
	\begin{equation}\label{ratopt}
	p^\ast = \inf_{\textbf{x} \in K} \frac{p(\textbf{x})}{q(\textbf{x})},
	\end{equation}
	where $q,p \in \mathbb{R[\textbf{x}]}$ are relatively prime and we may assume $q(\textbf{x}) > 0$ for all $\textbf{x} \in K$.
	Indeed, if $q$ changes signs on $K$, Jibetean and De Klerk \cite[Corollary 1]{jibetean} showed that $p^\ast = -\infty$.
	We will in fact make the stronger assumption that $q(\textbf{x}) \ge 1$ on $K$, i.e.\ that we know a positive lower bound on the minimum of $q$ over $K$.
	The optimization problem (\ref{ratopt}) can be modeled as a GMP:
	
	\begin{equation}\label{ratGMP}
	\mathrm{val} = \inf_{\mu \in \mathcal{M}(K)_+} \left \{ \int_{K} p(\textbf{x}) \mathrm{d}\mu(\textbf{x}) : \int_{K} q(\textbf{x})\mathrm{d}\mu(\textbf{x}) =1 \right \}.
	\end{equation}
	
	The inequality constraint $\int_{K}\mathrm{d}\mu(\textbf{x}) \le 1$ is redundant if $q(\textbf{x})\ge1 \; \forall \textbf{x} \in K$ and
	can be added to obtain a problem of form (\ref{primal}). 
	
	
	We emphasize that minimizing a quadratic polynomial over the simplex $\Delta_{n-1}$ is already NP-hard, since it contains the 
	problem of computing  the size ($\alpha(G)$) of a maximum stable set of a graph $G$. Indeed,
	for a graph $G$ with adjacency matrix adjacency matrix $A$, Motzkin and Strauss
	\cite{MotzkinStrauss} showed that
	\[ \frac{1}{\alpha(G)} = \min_{\textbf{x} \in \Delta_{n-1}} \textbf{x}^T(A + I)\textbf{x}, \]
	where $I$ is the identity matrix,
	which is a quadratic polynomial optimization problem over the simplex.
	
	Similarly, deciding convexity of a homogeneous polynomial $f$ of degree $4$ or higher is known to be NP-hard \cite{convexity}. It can be modeled as polynomial optimization problem over the sphere. A homogeneous polynomial $f$ is convex if and only if
	\[ \min_{(\textbf{x},\textbf{y}) \in \mathcal{S}^{2n-1}} \textbf{y}^T \nabla f(\textbf{x}) \textbf{y} \ge 0,  \]
	which in turn be cast as a GMP over the sphere.

	\textbf{Polynomial cubature.} Another application that goes beyond polynomial optimization is concerned with polynomial cubature rules,
	see e.g. \cite{dunkl}, \cite{jameson}. Finding polynomial cubature rules is NP-hard in general, see \cite{ruy}.
	Let $N \in \mathbb{N}$. Consider the problem of multivariate numerical integration of a function $f$ over a set $K$ with
	respect to a given (reference) measure $\mu_0  \in \mathcal{M}(K)_+$. Loosely speaking, a cubature scheme consists of
	a set of nodes $\textbf{x}^{(\ell)} \in K$  and weights $\omega_\ell \ge 0$ for $\ell \in [N]$, respectively, such that
	
	\[ \int_{K} f(\textbf{x}) \mathrm{d}\mu(\textbf{x}) \approx \sum_{\ell = 1}^N\omega_\ell f(\textbf{x}^\ell).  \]
	
	A possibility to mitigate the error in this scheme is to choose the weights and points such that the approximation is exact
	for polynomials up to some fixed degree. The problem of finding such weights and nodes can be cast as a GMP.
	Let $d \in \mathbb{N}$ and $\beta \in \mathbb{N}^n$ any vector such that $\abs{\beta} > d$. Assume the reference measure $\mu_0$ is a probability measure, otherwise set $\mu_0 \leftarrow \mu_0 / \mu_0(K)$. In the GMP given by
	
	\begin{equation}\label{cubature}
	\begin{aligned}
	\mathrm{val} := \inf_{\mu \in \mathcal{M}(K)_+} & \int_{K} \textbf{x}^\beta \mathrm{d}\mu(\textbf{x})  \\
	\text{s.t. } & \int_{K} \textbf{x}^\alpha \mathrm{d}\mu(\textbf{x}) =  \int_{K} \textbf{x}^\alpha \mathrm{d}\mu_0(\textbf{x}) \; \forall \alpha \in \mathbb{N}^n_d \\
	\end{aligned}
	\end{equation}
	the redundant constraint $\int_{K}\mathrm{d}\mu(\textbf{x})\le 1$ can be added to turn it into a GMP of form (\ref{primal}). 
	The solution $\mu^\ast$ to (\ref{cubature}) will be of the form $\mu^\ast = \sum_{\ell = 1}^N \omega_\ell \delta_{x^{(\ell)}}$, 
	where $N \le \abs{\mathbb{N}_d^n} = {{n+d}\choose{d}}$ by Theorem \ref{atomic}. This result is known as Tchakaloff's theorem \cite{tchak}.
	There is some freedom in the choice of the objective function, however, note that it should be 
	linearly independent of $\{ \textbf{x}^\alpha \}$ for $\alpha \in \mathbb{N}_d^n$. Hence, 
	our approach discussed in this paper may be applied to the problem of finding cubature rules for  measures on the simplex or sphere.

	\section{A linear relaxation hierarchy over the simplex}\label{chapterLP}
	
	A moment sequence $(y_\alpha)_{\alpha \in \mathbb{N}^n} \subset \mathbb{R}$ of a measure $\mu \in \mathcal{M}(K)$ is an infinite sequence such that
	\[ y_\alpha = \int_K \textbf{x}^\alpha \mathrm{d}\mu(\textbf{x}) \; \forall \alpha \in \mathbb{N}^n. \]
	Let $L : \mathbb{R[\textbf{x}]} \rightarrow \mathbb{R}$ be a linear operator
	\[ p(\textbf{x}) = \sum_{\alpha \in \mathbb{N}^n} p_\alpha\textbf{x}^\alpha \mapsto L(p) = \sum_{\alpha \in \mathbb{N}^n} p_\alpha y_\alpha  \]
	that maps monomials to their respective moments.
	Thus, to an optimal solution $\mu^\ast$ of a GMP there is an associated linear
	functional $L^\ast$ such that $L^\ast(f_0) = \text{val}$ and $L^\ast(f_i)= b_i$ for all $i \in [m]$ as well as $L^\ast(1)\le1$. The idea of
	the relaxation we are about to introduce is to approximate the optimal solution by a sequence (hierarchy) of linear functionals $L^{(r)}$ that
	depend on $ r = 1, 2, \dots$. Let $K = \Delta_{n-1}$. For $i = 0, 1, \dots, m$ let $f_i$ be a real homogeneous polynomial of degree $d$ and let $r \ge d$.
	Consider the following linear relaxation of (\ref{primal}):
	\begin{equation}\label{relax}
	\begin{aligned}	
	\underline{f}_{\text{LP}}^{(r)} = \min \quad & L^{(r)}(f_0) \\ 
	\text{s.t.} \quad & L^{(r)}(f_i)  = b_i \quad \forall \; i \in [m] \\
	& L^{(r)}(1) \le 1 \\
	&	L^{(r)}(\textbf{x}^\alpha)  \ge 0 \quad \forall \;  \vert \alpha \vert \le r \\
	& L^{(r)}(\textbf{x}^\alpha)  = L^{(r)} \left ( \textbf{x}^{\alpha} \sum_{i=1}^{n} x_i  \right ) \quad \forall \;\vert \alpha \vert \le r-1.
	\end{aligned}
	\end{equation}
	Every feasible solution $\mu^\prime$ to (\ref{primal}) provides an upper bound for (\ref{relax}) by setting $L^{(r)}(\textbf{x}^\alpha) = \langle \textbf{x}^\alpha, \mu^\prime \rangle$. Hence, $\underline{f}_{\text{LP}}^{(r)} \le \text{val} $. To see it is a linear program (LP) note that each $L^{(r)}(\textbf{x}^\alpha)$ can be replaced by a scalar variable $y_\alpha$ and the resulting program is an LP.
	The second last constraint is reflecting the necessary condition for a positive measure $\mu$ over the simplex:
	\[ \langle \textbf{x}^\alpha, \mu \rangle = \int_{\Delta_{n-1}} \textbf{x}^\alpha \mathrm{d}\mu \ge 0 \quad \forall \alpha \in \mathbb{N}^n. \]
	The last constraint in \eqref{relax} arises from the fact that
	\[ L^{(r)}(p)=L^{(r)}(q)  \text{ if } p(\textbf{x}) = q(\textbf{x}) \quad \forall \textbf{x} \in \Delta_{n-1}. \]
	Equivalently, defining the
	ideal $\mathcal{I} = \{ \textbf{x} \mapsto p(\textbf{x})\left( 1 - \sum_{i=1}^{n} x_i \right) \; : \; p \in \mathbb{R}[\textbf{x}] \}$ 
	we
	require
	\[ L^{(r)}(p)=L^{(r)}(q) \Leftrightarrow  p = q \mod \mathcal{I}.
	\]
	
	We state two lemmas that will come in handy in our later analysis.
	\begin{lemma}\label{lemma}
		Let $r,k \in \mathbb{N}$ with $k\le r$ and let $L^{(r)}$ be a feasible solution to the linear relaxation (\ref{relax}) for some $f_0, f_1, \dots, f_m$. Then for all $\normalfont{\textbf{x}}^{\gamma}$ with $\gamma \in \mathbb{N}^n $ and $\vert \gamma \vert \le r-k$ we have
		\[ L^{(r)}\left( \normalfont{\textbf{x}}^\gamma  \right )  = L^{(r)} \left ( \normalfont{\textbf{x}}^{\gamma} \left( \sum_{i=1}^{n} x_i  \right)^k \right ). \]
	\end{lemma}
	\begin{proof}
		
			The last equality constraint in the relaxation forces
			\[  L^{(r)}(\textbf{x}^\alpha ) = L^{(r)} \left ( \textbf{x}^{\alpha} \sum_{i=1}^{n} x_i  \right ) \quad \forall \enskip \vert \alpha \vert \le r-1. \]
			Therefore, noting that $\textbf{x}^{e_j}= x_j$ we have
			\begin{align*}
			L^{(r)}(\textbf{x}^\beta \textbf{x}^{e_j})  & = L^{(r)} \left ( \textbf{x}^{\beta} \textbf{x}^{e_j} \sum_{i=1}^{n} x_i  \right ) \quad \forall \enskip \vert \beta \vert \le r-2 \\
			\Rightarrow \sum_{j =1}^{n} L^{(r)}(\textbf{x}^\beta \textbf{x}^{e_j}) & = \sum_{j =1}^{n} L^{(r)} \left ( \textbf{x}^{\beta} \textbf{x}^{e_j} \sum_{i=1}^{n} x_i  \right ) \quad \forall \enskip \vert \beta \vert \le r-2 \\
			\Leftrightarrow L^{(r)}\left( \textbf{x}^\beta \sum_{j =1}^{n} \textbf{x}^{e_j} \right) & = L^{(r)} \left ( \textbf{x}^{\beta} \sum_{j =1}^{n} \textbf{x}^{e_j} \sum_{i=1}^{n} x_i  \right ) \quad \forall \enskip \vert \beta \vert \le r-2 \\
			&= L^{(r)} \left ( \textbf{x}^{\beta} \left( \sum_{i=1}^{n} x_i \right)^2 \right ) \quad \forall \enskip \vert \beta \vert \le r-2.
			\end{align*}
			Hence,
			\[  L^{(r)}(\textbf{x}^\beta ) =  L^{(r)}\left( \textbf{x}^\beta \sum_{i=1}^{n} x_i \right )  = L^{(r)} \left ( \textbf{x}^{\beta} \left( \sum_{i=1}^{n} x_i  \right)^2 \right ) \quad \forall \quad \vert \beta \vert \le r-2. \]
			Reiterating this procedure leads us to the desired outcome.
	\end{proof}
	
	\begin{lemma}\label{lemma2}
		Consider the GMP given in (\ref{primal}) and let $(y,t) \in \mathbb{R}^m \times \mathbb{R}_+$. Then the pair $(y,t)$ is dual optimal only if
		\[ 0 = \min_{\normalfont{\textbf{x}} \in K} \left(f_0(\normalfont{\textbf{x}}) - \sum_{i=1}^m y_i f_i(\normalfont{\textbf{x}})+t \right). \]
	\end{lemma}
	
	\begin{proof}
			The minimization problem
			
			\[ \min_{\textbf{x} \in K} \left( f_0(\textbf{x}) - \sum_{i = 1}^{m} y_if_i(\textbf{x}) +t \right) \]
			
			is equivalent to
			\begin{equation}\label{equLemma} \inf_{\mu \in \mathcal{M}(K)_+}\left\{\int_{K}f_0(\textbf{x}) - \sum_{i = 1}^{m} y_if_i(\textbf{x})\mathrm{d}+t \mathrm{d}\mu(\textbf{x}) : \int_{K}\mathrm{d}\mu = 1 \right\}. \end{equation}
			
			By Theorem \ref{cor1} there is no duality gap and there exists a primal optimal solution $\mu^\ast$ to the GMP (\ref{primal}). Set $\nu = \mu^\ast / \mu^\ast(K)$. Hence, $\nu$ is a probability measure and therefore a feasible solution to (\ref{equLemma}). We deduce
			\begin{align*}
			0 & \le  \min_{\textbf{x} \in K} f_0(\textbf{x}) - \sum_{i = 1}^{m} y_if_i(\textbf{x})+t \\
			&\le \int_{K} f_0(\textbf{x}) - \sum_{i = 1}^{m} \bar{y}_if_i(\textbf{x})+t \mathrm{d}\nu(\textbf{x}) \\
			&= \frac{1}{\mu^\ast(K)}\left( \int_{K} f_0(\textbf{x})\mathrm{d}\mu^\ast(\textbf{x}) -\sum_{i = 1}^{m}y_i \int_{K}f_i(\textbf{x})\mathrm{d}\mu^\ast(\textbf{x})+ t\int_K \mathrm{d}\mu^\ast(\textbf{x}) \right) \\
			&\le \frac{1}{\mu^\ast(K)} \left( \text{val}- y^Tb +t\right ) = 0,
			\end{align*}
			where the first inequality follows from the definition of the dual (\ref{dual}) of the GMP and the last equality from strong duality. 
	\end{proof}
	
	When we consider the case where $K = \Delta_{n-1}$, we may, without 
	loss of generality, assume the $f_i$ to be homogeneous of the same degree for all $i = 0, 1, \dots, m$.
	Indeed, let $f(\textbf{x}) = \sum_{j = 0}^d f_j(\textbf{x})$, where $\text{deg}(f_j)=j$. Then, $g(\textbf{x}) := \sum_{j=0}^d f_j(\textbf{x})\left(\sum_{i=1}^n x_i \right)^{d-j}$ is homogeneous of degree $d$ and $f(\textbf{x})=g(\textbf{x})$ for all $\textbf{x} \in \Delta_{n-1}$.

	\subsection{Convergence analysis}
	The following theorem is a refinement of a result by Powers and Reznick \cite{powersreznick}, obtained by de Klerk, Laurent and Parrilo \cite[Theorem 1.1]{PTAS}. It is a quantitative version of P\'olya's Positivstellensatz (see, e.g. \cite{reznicksurvey} for a survey), and it will be crucial in our analysis of the simplex case.
	\begin{thm}\label{theosimplex}
		Suppose $f \in \mathbb{R}[\normalfont{\textbf{x}}]$ is a homogeneous polynomial of degree $d$ of the form
		$ f(\normalfont{\textbf{x}}) = \sum_{\vert \alpha \vert = d} f_{\alpha}\normalfont{\textbf{x}}^\alpha . $
		Let $\varepsilon = \min_{\Delta_{n-1}} f(\normalfont{\textbf{x}})$ and define
		\begin{equation}\label{b(f)}
		B(f)= \max_{\vert \alpha \vert = d} \frac{\alpha_1 ! \dots \alpha_n ! }{d!} f_{\alpha}.
		\end{equation}Then the polynomial $(x_1 + \dots + x_n)^k f(\normalfont{\textbf{x}}) $ has only positive coefficients if
		\begin{equation}\label{k} k > \frac{d(d-1)}{2} \frac{B(f)}{\varepsilon}-d. \end{equation}
	\end{thm}
	
	We continue by stating and proving one of the main results of this paper.
	
	\begin{thm}\label{LPtheorem}
		Let {\normalfont val} be the optimal value of the GMP (\ref{primal}) for input data $K = \Delta_{n-1}, f_0, f_1, \dots, f_m \in \mathbb{R}[\normalfont{\textbf{x}}]$ homogeneous of degree $d$ and $b_1, \dots, b_m \in \mathbb{R}$. Assume there exists a dual optimal solution $(\bar{y},t)$ and let $f_{m+1}(\normalfont{\textbf{x}}):=1$ for every $\normalfont{\textbf{x}} \in \Delta_{n-1}$ and set $\bar{y}_{m+1}=-t$. Then, setting $y_0 = 1$ and $y_i = -\bar{y}_i$ for $i \in [m+1]$ we have
		\begin{equation}\label{thmequ}
		0 \le \text{{\normalfont val}} - \underline{f}_{\text{\normalfont{LP}}}^{(r)} \le \frac{\left( \sum_{i = 0}^{m+1}  B(y_if_i) +t \right) d(d-1)}{2(r-1)-d(d-1)}, \end{equation}
		for $B(\cdot)$ as in (\ref{b(f)}) and $r > d(d-1)/2+1$.
	\end{thm}
	
	\begin{proof}
		By Theorem \ref{cor1}  there is no duality gap.
		Let $r > d(d-1)/2+1$ and let $L^{(r)}$ be an optimal solution to (\ref{relax}). Fix some $\varepsilon>0$. Then,
		
		\begin{align*}
		0 \le \text{val}-\underline{f}_{\text{LP}}^{(r)}  &= \text{val} - L^{(r)}\left( \sum_{i = 1}^m \bar{y}_if_i-t + f_0 - \sum_{i=1}^m \bar{y}_if_i +t \right) \\
		&= \text{val} - \sum_{i=1}^{m}\bar{y}_i L^{(r)}(f_i)+tL^{(r)}(1) -L^{(r)}\left( f_0 - \sum_{i=1}^m \bar{y}_if_i +t \right) \\
		& \le  \text{val} - \sum_{i=1}^{m}\bar{y}_i b_i+t  -L^{(r)} \left( f_0 - \sum_{i=1}^m \bar{y}_if_i +t\right)  \\
		&= -L^{(r)}\left( f_0 - \sum_{i=1}^m \bar{y}_if_i +t  \right) \\
		& =  - L^{(r)}\left( f_0 - \sum_{i=1}^m \bar{y}_if_i +t+\varepsilon \right) +\varepsilon L^{(r)}(1) \\
		& \le - L^{(r)}\left( f_0 - \sum_{i=1}^m \bar{y}_if_i +t+\varepsilon \right) +\varepsilon,
		\end{align*}
		
		where both inequalities follow from the fact that $L^{(r)}(1)\le 1$. By Lemma \ref{lemma2} we have
		$ \min_{\textbf{x} \in \Delta_{n-1}} f_0(\textbf{x}) - \sum_{i = 1}^{m+1} \bar{y}_if_i(\textbf{x})  + \varepsilon = \varepsilon$.
		We assume wlog that $f_0 - \sum_{i=1}^{m+1} \bar{y}_if_i$ is homogeneous of degree $d$. Define
		\[
		f := f_0 - \sum_{i=1}^{m+1} \bar{y}_if_i + \varepsilon \left(\sum_{i=1}^{n} x_i \right)^d,
		\]
		which is homogeneous as well and its minimum over the simplex is $\varepsilon$.
		Hence, by Theorem \ref{theosimplex} for $k$ as in (\ref{k}) we have
		\[
		f(\textbf{x}) \left(\sum_{i=1}^{n} x_i \right)^k = \sum_{\beta \in \mathbb{N}^n_{d+k}}c_{\beta}x^{\beta}
		\]
		with $c_{\beta}> 0$ for all $\beta \in \mathbb{N}^n_{d+k}$. To determine the smallest integer $k$ for which the theorem holds we will bound $B(f)$. For this, set $y_0 = 1$ and $y_i = -\bar{y}_i$. We may rewrite $f$ as
		\begin{align*}
		f &= \sum_{i = 0}^{m+1} y_i f_i +\varepsilon  \left(\sum_{i=1}^{n} x_i \right)^d \\
		&= \sum_{i=0}^{m+1} y_i f_i +\varepsilon  \left( \sum_{\vert \alpha \vert = d}\binom{d}{\alpha_1 \dots \alpha_n } x^\alpha \right) \\
		& = \sum_{\vert \alpha \vert = d} \left( \sum_{i=0}^{m+1} y_i f_{i,\alpha} +\varepsilon \binom{d}{\alpha_1 \dots \alpha_n }\right)  x^\alpha.
		\end{align*}
		Then,
		\begin{align*}
		B(f) & = \max_{ \alpha }\left[ \left ( \sum_{i=0}^{m+1} y_i f_{i,\alpha} + \frac{d!}{\alpha_1 ! \dots \alpha_n !} \varepsilon    \right)\frac{\alpha_1 ! \dots \alpha_n !}{d!} \right] \\
		& = \left( \max_{ \alpha } \left( \sum_{i=0}^{m+1} y_i f_{i,\alpha} \right) \frac{\alpha_1 ! \dots \alpha_n !}{d!} \right) +\varepsilon   \\
		& \le \sum_{i=0}^{m+1}  \left( \max_{ \alpha } y_if_{i,\alpha} \frac{\alpha_1 ! \dots \alpha_n !}{d!} \right) +\varepsilon  \\
		& = \sum_{i=0}^{m+1}  B(y_if_i)+\varepsilon  .
		\end{align*}
		
		If $r$ is large enough, i.e.,
		\[ r \ge  \left  \lceil \frac{d(d-1)}{2}\frac{\sum_{i=0}^{m+1}  B(y_if_i)+\varepsilon  }{\varepsilon} \right \rceil   \ge  \left  \lceil \frac{d(d-1)}{2}\frac{B(f)}{\varepsilon} \right \rceil,  \]
		it follows from Lemma \ref{lemma} that
		\begin{align*}
		- L^{(r)}\left(f_0 - \sum_{i=1}^{m+1} \bar{y}_if_i + \varepsilon  \right) + \varepsilon & =  \varepsilon  - L^{(r)}(f)  \\
		& =  \varepsilon  - L^{(r)}\left(f \left( \sum_{i=1}^{n} x_i \right)^k \right) \\
		& = \varepsilon  - L^{(r)} \left (\sum_{\beta \in \mathbb{N}_{k+d}^n} c_{\beta} x^{\beta} \right )  \le \varepsilon,
		\end{align*}
		where the last inequality follows from the fact that $L^{(r)}(\textbf{x}^\alpha) \ge 0$ for all $\vert \alpha \vert \le r$.
		One may bound $r$ as follows
		\begin{align*}
		r-1 & \le \frac{d(d-1)}{2}\left( \frac{\sum_{i =0}^{m+1} B(y_if_i)}{\varepsilon} + 1 \right) \\
		\Leftrightarrow \varepsilon & \le \frac{ \sum_{i=0}^{m+1} B(y_if_i) d(d-1)}{2(r-1)-d(d-1)},
		\end{align*}
		concluding the proof. 
	\end{proof}

	\section{Lasserre hierarchy over the sphere}
	\label{sec:Lasserre on sphere}
	
	We now consider the GMP \eqref{primal} over the sphere, i.e.\ we consider the case $K= \mathcal{S}^{n-1}$.
	Additionally, we assume the $f_0, f_1, \dots, f_m$ in \eqref{primal} are homogeneous polynomials of even degree $2d$.
	
	The Lasserre hierarchy \cite{lasserre3} of semidefinite relaxations of the GMP \eqref{primal} over the sphere  is given by
	
	\begin{equation}\label{SDPrelax2}
	\begin{aligned}
	\underline{f}_{\text{SDP}}^{(2r)} = \min & \; L^{(2r)}(f_0)  \\ 
	\text{s.t.} \quad& L^{(2r)}(f_i ) = b_i \quad \forall i \in [m] \\
	& L^{(2r)}(1) \le 1 \\
	&L^{(2r)}\left( [\textbf{x}]_{r} [\textbf{x}]_{r}^T\right)  \succeq 0 \quad \\
	& L^{(2r)}(\textbf{x}^\alpha)  = L^{(2r)} \left ( \textbf{x}^{\alpha} \norm{\textbf{x}}_2^2 \right ) \quad \forall \quad \vert \alpha \vert \le 2r-2,
	\end{aligned}
	\end{equation}
	where the $L^{(2r)}$ operator is now applied entry-wise to matrix-valued functions, where needed.
	
	The following lemma enables us to use  a quantitative Positivstellensatz by Fang an Fawzi \cite{fang} for positive polynomials on the sphere, to obtain
	a rate of convergence of the Lasserre hierarchy.
	
	\begin{lemma}\label{lemma}
		Let $L : \mathbb{R}[\normalfont{\textbf{x}}]_{2k} \rightarrow \mathbb{R}$ be a linear operator and suppose $L\left([\normalfont{\textbf{x}}]_k [\normalfont{\textbf{x}}]_k^T \right) \succeq 0$, where the operator is applied entrywise to the matrix $[\normalfont{\textbf{x}}]_k [\normalfont{\textbf{x}}]_k^T$. Then, $L(\sigma) \ge 0$ for all $\sigma \in \Sigma[\normalfont{\textbf{x}}]_k$.
	\end{lemma}
	
	\begin{proof}
		Let $\sigma \in \Sigma[\textbf{x}]_k$ be a sum of squares of degree $2k$. Then there exists $A \succeq 0$ such that $\sigma = [\textbf{x}]_k^TA[\textbf{x}]_k$. Let $\langle \cdot, \cdot \rangle$ denote the trace inner product. We have
		\begin{align*}
		L(\sigma) & = L\left( [\textbf{x}]_k^T A [\textbf{x}]_k  \right) \\
		&= L \left( \langle A, [\textbf{x}]_k [\textbf{x}]_k^T\rangle \right) \\
		& = L \left( \sum_{i,j} A_{i,j} ([\textbf{x}]_k)_i( [\textbf{x}]_k)_j \right) \\
		& = \sum_{i,j} A_{i,j} L\left(([\textbf{x}]_k)_i( [\textbf{x}]_k)_j \right) \\
		&= \langle A, L\left([\textbf{x}]_k [\textbf{x}]_k^T \right) \rangle \ge 0,
		\end{align*}
		since both $A$ and $L\left([\textbf{x}]_k [\textbf{x}]_k^T \right)$ are psd.
	\end{proof}
	
	The quantitative Positivstellensatz by Fang and Fawzi \cite{fang} is as follows.
	
	\begin{thm}\cite[Theorem 3.8]{fang}\label{fawzi}
		Assume $f$ is a homogeneous polynomial of degree $2d$ such that $0 \le f(\normalfont{\textbf{x}}) \le 1$ for all $\normalfont{\textbf{x}} \in \mathcal{S}^{n-1}$ and $d \le n$. There are constants $C_d, C_d^\prime$ that depend only on $d$ such that if $r \ge C_d n$ then
		\[ f + C_d^\prime(d/r)^2= \sigma(\normalfont{\textbf{x}}) + (1-\norm{\normalfont{\textbf{x}}}_2^2)h(\normalfont{\textbf{x}}) \]
		for $\sigma(\normalfont{\textbf{x}}) \in \Sigma[\normalfont{\textbf{x}}]_r$  and $h \in \mathbb{R}[\normalfont{\textbf{x}}]_{2r-2}$.
	\end{thm}
	
	We may now use the theorem by Fang and Fawzi \cite{fang} and Lemma \ref{lemma} 
	to derive a rate of convergence for Lasserre hierarchy \cite{lasserre3} of the GMP on the sphere as follows.
	
	\begin{thm}
		Let $\mathrm{val}$ be the optimal value of the GMP (\ref{primal}) for input data $K = \mathcal{S}^{n-1}, f_0, f_1, \dots, f_m \in \mathbb{R}[\normalfont{\textbf{x}}]$ homogeneous of even degree $2d$, $b_1, \dots, b_m \in \mathbb{R}$ and $d \le n$. Let $(\bar{y},t)$ be a dual optimal solution and let $f_{m+1}(\normalfont{\textbf{x}}):=1$ for every $\normalfont{\textbf{x}} \in \mathcal{S}^{n-1}$, set $\bar{y}_{m+1}=-t$ and set $y_0 = 1$ and $y = -\bar{y}$. Further, let $f^{i,y_i}_{\max} = \max_{\normalfont{\textbf{x}}\in \mathcal{S}^{n-1}} y_if_i(\normalfont{\textbf{x}})$. There exist constants $C_d, C_d^\prime$, only dependent on $d$, such that if $r \ge C_d n$ we have
		\[
		0 \le \mathrm{val}-\underline{f}_{\mathrm{SDP}}^{(2r)} \le  \frac{C_d^\prime d^2 \sum_{i=0}^{m+1} f^{i,y_i}_{\max}}{r^2}.	\]
	\end{thm}
	\begin{proof}
		
			As in the proof of Theorem \ref{LPtheorem}, Theorem \ref{cor1} gives us strong duality. Let $r \ge C_dn$ and let $L^{(2r)}$ be an optimal solution to (\ref{SDPrelax2}). Then by the same reasoning as in Theorem \ref{LPtheorem},
			
			\[ 0 \le \text{val}-\underline{f}_{\text{SDP}}^{(2r)} \le -L^{(2r)}\left( f_0 - \sum_{i=1}^{m+1} \bar{y}_if_i  \right). \]
			
			Set $f := f_0 - \sum_{i = 1}^{m+1} y_if_i $ and $f_{\max} = \max_{\textbf{x}\in \mathcal{S}^{n-1}} f(\textbf{x})$. Then $\tilde{f} = f/f_{\max}$ satisfies $0 \le \tilde{f} \le 1$ by Lemma \ref{lemma2}. We find for any $\delta \ge 0$
			\begin{align*}
			-L^{(2r)}\left( f_0 - \sum_{i=1}^{m+1} \bar{y}_if_i \right) & = -f_{\max} L^{(2r)}\left( \tilde{f} \right) \\
			& \le -f_{\max} L^{(2r)}\left( \tilde{f} +\delta \right) + \delta f_{\max}.
			\end{align*}
			
			Choosing $\delta = \frac{C_d^\prime d^2}{r^2}$ and applying Theorem \ref{fawzi} we see that $\tilde{f}+ \delta = \sigma + (1-\norm{\textbf{x}}_2^2)h$ for $\sigma \in \Sigma[\textbf{x}]_r$ and $h \in \mathbb{R}[\textbf{x}]_{2r-2}$.
			
			Thus, since $L^{(2r)}(\textbf{x}^\alpha) = L^{(2r)}(\textbf{x}^\alpha \norm{\textbf{x}}_2^2)$ we have
			\begin{align*}
			-f_{\max} L^{(2r)}\left( \tilde{f} +\frac{C_d^\prime d^2}{r^2} \right) + \frac{C_d^\prime d^2}{r^2}f_{\max}&= -f_{\max} L^{(2r)}\left( \sigma + (1-\norm{\textbf{x}}_2^2)h \right) + \frac{C_d^\prime d^2}{r^2} f_{\max}\\
			& =  -f_{\max} L^{(2r)}\left( \sigma \right)  + \frac{C_d^\prime d^2}{r^2} f_{\max}\\
			& \le C_d^\prime\frac{d^2}{r^2}f_{\max},
			\end{align*}
			where the last inequality follows from Lemma \ref{lemma}.
			Noting that
			\[ f_{\max} = \max_{\textbf{x}\in \mathcal{S}^{n-1}} \left( f_0(\textbf{x}) - \sum_{i = 1}^{m+1} \bar{y}_i f_i(\textbf{x})\right) \le \sum_{i=0}^{m+1} \max_{\textbf{x}\in \mathcal{S}^{n-1}} y_i f_i(\textbf{x}) = \sum_{i=0}^{m+1} f_{\max}^{i,y_i} \]
			we arrive at the result.

	\end{proof}
	
	\section{Limiting behavior of the hierarchies of linear operators}
	\label{sec:limiting behavior}
	Consider the case when $K = \Delta_{n-1}$. When looking at the linear operators in the relaxation hierarchies (\ref{relax})  one would expect that in the limit, i.e. for $r \rightarrow \infty$, the operators $L^{(r)}(\cdot)$ behave like $\langle \cdot, \mu \rangle$ for some positive measure $\mu$. In the rest of this section we prove that this is in fact the case and we will define the limit in a meaningful way.
	Consider again the ideal $\mathcal{I} = \{ \textbf{x} \mapsto p(\textbf{x})\left( 1- \sum_{i=1}^{n}x_i \right) : p \in \mathbb{R[\textbf{x}]} \}$ and let $L : \mathbb{R[\textbf{x}]} / \mathcal{I} \rightarrow \mathbb{R}$ be a linear operator such that
	\begin{enumerate}
		\item $L(\textbf{x}^\alpha) \ge 0$ for all $\alpha \in \mathbb{N}^n$
		\item $L(1)\le1$
	\end{enumerate}
	and let
	\[ \mathcal{L} = \{ L : \mathbb{R}[\textbf{x}]/\mathcal{I} \rightarrow \mathbb{R} : L \text{ fulfills conditions } 1. \text{ and } 2. \} \]
	be the class of all linear operators that satisfy the conditions above. Note that for every $L \in \mathcal{L}$ the relation
	\[ L\left(\left(1-\sum_{i=1}^{n}x_i\right)\textbf{x}^\alpha \right) = 0  \text{ for all } \alpha \in \mathbb{N}^n \]
	trivially holds.
	If $\norm{f} = \sup_{\textbf{x} \in \Delta_{n-1}} \vert f(\textbf{x}) \vert$, then $(\mathbb{R[\textbf{x}]}/\mathcal{I}, \norm{\cdot})$ is a normed vector space.
	\begin{thm}(see, e.g. \cite[Theorem 1.4.2]{megg})\label{continuity}
		Suppose $F: X \rightarrow Y$ is a linear operator between two normed vector spaces $(X, \norm{\cdot}_X)$ and $(Y, \norm{\cdot}_Y)$, then the following are equivalent
		\begin{enumerate}
			\item $F$ is continuous
			\item $\norm{Fx}_Y \le M \norm{x}_X$ for some $M\in \mathbb{R}$.
		\end{enumerate}
	\end{thm}
	
	Using Theorem \ref{continuity} we can prove that the operators we consider are continuous in the limit.
	
	\begin{lemma}\label{lemmaCont}
		Every $L \in \mathcal{L}$ is continuous.
	\end{lemma}
	\begin{proof}
		By Theorem \ref{continuity} it suffices to show that every $L \in \mathcal{L}$ satisfies
		\[ \abs{L(f)} \le M \norm{f} = M \sup_{\textbf{x} \in \Delta_{n-1}} \abs{f} \]
		for all $f \in \mathbb{R[\textbf{x}]}/\mathcal{I}$. Hence, let $f \in \mathbb{R}[\textbf{x}]/\mathcal{I}$ and let $\norm{f} = \sup_{\textbf{x} \in \Delta_{n-1}} \vert f(\textbf{x}) \vert$. Also set
		\[f_{\min} = \min_{x \in \Delta_{n-1}} f(\textbf{x}) \ge -\norm{f} \text{ and } f_{\max} = \max_{x \in \Delta_{n-1}} f(\textbf{x}) \le \norm{f}. \]
		Let $L^\ast$ be the optimizer of
		\[\min L(f) \text{ s.t. } L \in \mathcal{L} \]
		and note that $L^\ast (f) = f_{\min}$ as an immediate consequence of Theorem \ref{LPtheorem}. Hence, for all $L \in \mathcal{L}$ we have
		\[ L(f) \ge L^\ast(f) = f_{\min} \ge -\norm{f}. \]
		Similarly, let $L^\prime$ be the optimizer of
		\[ \max L(f) \text{ s.t. } L \in \mathcal{L}. \]
		By the same reasoning we have $L^\prime(f) = f_{\max}$ and it follows that $L(f) \le \norm{f}$ for all $L \in \mathcal{L}$. Hence one can set $M=1$ and we see
		
	\[\abs{L(f)} \le \norm{f}. \] 
	\end{proof}
	
	The set $\mathbb{R[\textbf{x}]} / \mathcal{I}$ is dense in $\mathcal{C}(\Delta_{n-1})$. This means we can employ the following theorem in the next step.
	\begin{thm}(see, e.g. \cite[Theorem 1.9.1]{megg})\label{BLT}
		Suppose that $M$ is a dense subspace of a normed space $X$, that $Y$ is a Banach space, and that $T_0 : M \rightarrow Y$ is a bounded linear operator. Then there is a unique continuous function $T : X \rightarrow Y$ that agrees with $T_0$ on $M$. This function $T$ is a bounded linear operator and $\norm{T} = \norm{T_0}$.
	\end{thm}
	
	Now let
	\[\bar{\mathcal{L}} = \left\{ \bar{L} : \mathcal{C}(\Delta_{n-1}) \rightarrow \mathbb{R} : \bar{L} \text{ is the continuous linear extension  of some } L \in \mathcal{L} \right\}. \]
	
	\begin{prop}
		Let $\bar{L} \in \bar{\mathcal{L}}$ and $f \in \mathcal{C}(\Delta_{n-1})$. Then
		\[ \bar{L}(f) = \int_{\Delta_{n-1}} f(\normalfont{\textbf{x}})\mathrm{d}\mu(\normalfont{\textbf{x}}) \]
		for some positive measure $\mu$ supported on $\Delta_{n-1}$, satisfying $\mu(\Delta_{n-1}) \le 1$.
	\end{prop}
	\begin{proof}
		It is sufficient to show $\bar{L}(f) \ge 0$ for all $f \in \mathcal{C}(\Delta_{n-1})_+ = \{f \in \mathcal{C}(\Delta_{n-1}) : f(\textbf{x}) \ge 0 \; \forall \textbf{x} \in \Delta_{n-1} \}$. To see this, note that the space $\mathcal{C}(\Delta_{n-1}$) can be ordered by the convex cone $\mathcal{C}(\Delta_{n-1})_+$. Now $\bar{L}(f) \ge 0$ for all $f \in \mathcal{C}(\Delta_{n-1})_+$ implies that $\bar{L} \in \left( \mathcal{C}(\Delta_{n-1})_+ \right)^\ast$, i.e. the dual cone of $\mathcal{C}(\Delta_{n-1})_+$ which is known to be the set of finite Borel measures on $\Delta_{n-1}$. Let $f$ be a homogeneous continuous function that is non-negative on the simplex and consider its Bernstein approximation of order $r$ given by
		\[ \mathcal{B}_f^r(\textbf{x}) = \sum_{\substack{\alpha \in \mathbb{N}_r^n \\ \vert \alpha \vert = r }} f\left(\frac{\alpha}{r}\right){{r}\choose{\alpha}}. \]
		The approximation converges uniformly to $f$ as $r \rightarrow \infty$ since $f$ is continuous. Using Lemma \ref{lemmaCont} we see
		\begin{align*}
		\bar{L}(f) & = \bar{L}( \lim_{r \rightarrow \infty} \mathcal{B}_f^r ) \\
		& \overset{\bar{L}\text{ cont.}}{=} \lim_{r \rightarrow \infty} \bar{L}(\mathcal{B}_f^r) \\
		& = \lim_{r \rightarrow \infty} \sum\limits_{\substack{\alpha \in \mathbb{N}_r^n \\ \vert \alpha \vert = r }}  \underbrace{ f\left(\frac{\alpha_1}{r}, \dots, \frac{\alpha_n}{r}\right)}_{\ge0} \underbrace{ \binom{r}{\alpha}}_{\ge0} \underbrace{ \bar{L}(\textbf{x}^\alpha)}_{\ge0}  \ge 0.
		\end{align*}
		Hence, it follows that $\bar{L}(f) = \langle f, \mu \rangle$ for some positive measure $\mu$, such that $\mu(\Delta_{n-1}) \le 1$.
	\end{proof}
	
	\begin{rem}
		By the proof given above, it becomes clear that the continuous linear extension can 
		in fact be defined in terms of the limit of the Bernstein approximation, i.e., define $\bar{L}(f) := \lim_{r \rightarrow \infty} L(\mathcal{B}_f^r)$ for $f \in \mathcal{C}(\Delta_{n-1})$ and $L \in \mathcal{{L}}$.
	\end{rem}

	For the sphere case, i.e. $K = \mathcal{S}^{n-1}$ consider the following theorem.
	
	\begin{thm}(see, e.g. \cite[Theorem 3.8]{lasserre3})\label{repmeas}
		Let $\textbf{y} = (y_\alpha)_{\alpha \in \mathbb{N}^n} \subset \mathbb{R}^\infty$ be a given infinite real sequence, $L : \mathbb{R}[\normalfont{\textbf{x}}] \rightarrow \mathbb{R}$ be the linear operator defined by
		\[ p(\normalfont{\textbf{x}}) = \sum_{\alpha \in \mathbb{N}^n} p_\alpha \normalfont{\textbf{x}}^\alpha \mapsto L(p) = \sum_{\alpha \in \mathbb{N}^n} p_\alpha y_\alpha, \]
		and let $K = \{ \normalfont{\textbf{x}} \in \mathbb{R}^n : g_1(\normalfont{\textbf{x}})\ge0, \dots, g_m(\normalfont{\textbf{x}})\ge0 \}$. The sequence $\textbf{y}$ has a finite Borel representing measure with support contained in $K$ if and only if
		\[ L(f^2g_J) \ge0 \; \forall J \subseteq \{1, \dots, m\} \text{ and } f \in \mathbb{R}[\normalfont{\textbf{x}}], \]
		where
		$ g_J(\normalfont{\textbf{x}}) = \prod_{j \in J} g_j (\normalfont{\textbf{x}}). $
	\end{thm}
	
	Now, let $L$ be a linear operator such that
	
	\begin{enumerate}
		\item $L(1)\le1$
		\item $L([\textbf{x}]_t[\textbf{x}]_t^T) \succeq 0 \;\forall t \in \mathbb{N}$
		\item $L(\textbf{x}^\alpha) = L(\textbf{x}^\alpha \norm{\textbf{x}}_2^2) \; \forall \alpha \in \mathbb{N}^n$
	\end{enumerate}
	
	and let $\mathcal{L}^\prime = \{ L : \mathbb{R}[\textbf{x}] \rightarrow \mathbb{R} : L \text{ satisfies } 1. \text{ - } 3. \}$. Recall that as a 
	semialgebraic set the sphere can be written 
	as $\mathcal{S}^{n-1} = \{ \textbf{x} \in \mathbb{R}^n : g_1(\textbf{x}) := 1- \norm{\textbf{x}}_2^2 \ge 0, g_2(\textbf{x}): = \norm{\textbf{x}}_2^2-1 \ge 0 \}$.
	Then for $K = \mathcal{S}^{n-1}$ every $L \in \mathcal{L}^\prime$ satisfies all conditions of Theorem \ref{repmeas}. To see this,
	note that the only possibilities for $J$ are $\{ \emptyset, \{1\}, \{2\}, \{1,2\} \}$. Because of condition 3 we have 
	that $L(\pm(1-\norm{\textbf{x}}_2^2)p)=0$ for all $p \in \mathbb{R[\textbf{x}]}$ covering all cases except $J = \emptyset$. 
	For $J = \emptyset$ the condition reduces to $L(p^2)\ge 0$ which holds for all $p \in \mathbb{R[\textbf{x}]}$ because 
	of Lemma \ref{lemma}. Hence, every $L \in \mathcal{L}^\prime$ has a representing measure whose support is contained in $\mathcal{S}^{n-1}$.

	\section{Concluding remarks}\label{conclusion}
	
	In this last section we conclude by outlining the connection of our results to previous work. We show that
	--- in the special case of polynomial optimization on the simplex --- our RLT hierarchy reduces to one studied earlier by Bomze and De Klerk \cite{deklerk2},
	and De Klerk, Laurent and Parrilo \cite{PTAS}.

	De Klerk, Laurent and Parrilo \cite{PTAS} introduced the following hierarchy for minimizing a homogeneous polynomial $p \in \mathbb{R}[\textbf{x}]$ of degree $d$ over the simplex.
	
	\begin{equation}\label{pr}
	\begin{split}
	p^{(r)} = \max \lambda \text{ s.t. } & \text{ the polynomial } \left( \sum_{i=1}^n x_i \right)^r\left(p(\textbf{x})-\lambda \left( \sum_{i=1}^n x_i \right)^d \right) \\ &\text{ has only nonneg. coefficients.}
	\end{split}
	\end{equation}
	
	It was proved that $\lim_{r \rightarrow \infty} p^{(r)} = p_{\min} = \min_{\textbf{x} \in \Delta_{n-1}} p(\textbf{x})$.
	The LP hierarchy introduced in section \ref{chapterLP} of this paper is a generalization of the hierarchy \eqref{pr}, in the sense made precise in the following theorem.
	
	\begin{thm}\label{equiv}
		For some homogeneous polynomial $p \in \mathbb{R}[\normalfont{\textbf{x}}]$ of degree $d$ let $\underline{f}_{\text{\normalfont LP}}^{(r+d)}$ be the solution to the LP relaxation of the problem
		\[ \min_{\normalfont{\textbf{x}} \in \Delta_{n-1}}p(\normalfont{\textbf{x}}) = \text{\normalfont val} = \inf_{\mu \in \mathcal{M}(\Delta_{n-1})_+} \left\{ \int_{\Delta_{n-1}}p(\normalfont{\textbf{x}}) \mathrm{d}\mu(\normalfont{\textbf{x}}) : \int_{\Delta_{n-1}}\mathrm{d}\mu(\normalfont{\textbf{x}}) = 1   \right\}  \]
		for some $r \in \mathbb{N}$. Then,
		\[p^{(r)} = \underline{f}_{\text{\normalfont LP}}^{(r+d)}.\]
	\end{thm}
	
	\begin{proof}
			$"\le":$ Let $\lambda^\ast = p^{(r)}$ be optimal for (\ref{pr}). Then $\left( \sum_{i=1}^n x_i \right)^r\left(p(\textbf{x})-\lambda \left( \sum_{i=1}^n x_i \right)^d \right)$ has only negative coefficients and we find
			\begin{align*}
			0 &\le L^{(r+d)}\left( \left( \sum_{i=1}^n x_i \right)^r\left(p(\textbf{x})-\lambda^\ast \left( \sum_{i=1}^n x_i \right)^d \right) \right)\\
			&= L^{(r+d)}\left(\left(\sum_{i=1}^n x_i\right)^r p(\textbf{x})\right)	- \lambda^\ast L^{(r+d)}\left( \left( \sum_{i=1}^{n} x_i \right)^{r+d}\right) \\
			&= \underline{f}_{\text{\normalfont LP}}^{(r+d)} - \lambda^\ast
			\end{align*}
			for $L^{(r+d)}$ being the optimal solution to the LP relaxation. \\
			$"\ge":$
			For the multinomial coefficient \[ {{k}\choose{\alpha}} = {{k}\choose{\alpha_1, \dots, \alpha_n}} =  \frac{k!}{\alpha_1! \dots \alpha_n!}\] we define ${{k}\choose{\alpha}} = 0$ if $\alpha_i < 0$ for some $i \in [n]$.
			
			Consider the expansion
			\begin{align*}
			 \left( \sum_{i=1}^n x_i \right)^r\left(p(\textbf{x})-\lambda \left( \sum_{i=1}^n x_i \right)^d \right) 
			& =  \sum_{\vert \beta \vert = r} {{r}\choose{\beta}}\textbf{x}^\beta \sum_{\vert \alpha \vert = d} p_{\alpha}\textbf{x}^\alpha - \lambda \sum_{\vert \beta \vert =r+d} {{r+d}\choose{\beta}}\textbf{x}^\beta \\
			&= \sum_{\vert \beta \vert = r+d} \left(\sum_{\vert \alpha \vert = d}{{r}\choose{\beta-\alpha}} p_\alpha-\lambda {{r+d}\choose{\beta}} \right)\textbf{x}^\beta.
			\end{align*}
			Thus the LP formulation of (\ref{pr}) reads
			\begin{align*}
			p^{(r)} = \max & \enskip \lambda \\
			\text{s.t. }& {{r+d}\choose{\beta}}\lambda \le \sum_{\vert \alpha \vert = d} {{r}\choose{\beta-\alpha}}p_\alpha \quad \forall \vert \beta \vert = r+d
			\end{align*}
			with its dual
			\begin{align*}
			p^{(r)} = \min \enskip & \sum_{\vert \beta \vert= r+d}\sum_{\vert \alpha \vert = d} y_\beta {{r}\choose{\beta-\alpha}}p_\alpha \\
			\text{s.t. }& y_\beta \ge 0 \quad \forall \vert \beta \vert = r+d \\
			& \sum_{\vert \beta \vert = d+r}{{r+d}\choose{\beta}} y_\beta = 1.
			\end{align*}
			Let $y$ be an optimal solution for the dual and define
			\[ L^{(r+d)}(\textbf{x}^\beta) = y_\beta \quad \forall \vert \beta \vert = r+d. \]
			Then for $\vert \alpha \vert = r+d-1$ we let
			\[ L^{(r+d)}(\textbf{x}^\alpha) = \sum_{i = 1}^n y_{\alpha + e_i} \]
			and proceed in this manner for all $\vert \gamma \vert \le r+d-2$. The last constraint of the dual then implies
			\[ 1= \sum_{\vert \beta \vert = d+r}{{r+d}\choose{\beta}} y_\beta =  \sum_{\vert \beta \vert = d+r}{{r+d}\choose{\beta}} L^{(r+d)}(\textbf{x}^\beta) = L^{(r+d)}\left( \left(\sum_{i=1}^n x_i \right)^{r+d} \right).
			\]
			By construction we have
			\begin{enumerate}
				\item $L^{(r+d)}(\textbf{x}^\alpha) \ge 0$ for all $\vert \alpha \vert \le r+d$
				\item $L^{(r+d)}(\textbf{x}^\alpha) = L^{(r+d)}\left(\textbf{x}^\alpha \sum_{i=1}^n x_i \right)$ for all $\vert \alpha \vert \le r+d-1$
				\item $1 = L^{(r+d)}\left( \left(\sum_{i=1}^n x_i \right)^{r+d} \right) \overset{2.}{=} L^{(r+d)}(1).$
			\end{enumerate}
			Hence, the constructed solution for the LP relaxation is feasible. Further,
			\begin{align*}
			p^{(r)} &= \sum_{\vert \beta \vert= r+d}\sum_{\vert \alpha \vert = d} y_\beta {{r}\choose{\beta-\alpha}}p_\alpha \\
			& =  \sum_{\vert \beta \vert= r+d}\sum_{\vert \alpha \vert = d} L^{(r+d)}(\textbf{x}^\beta) {{r}\choose{\beta-\alpha}}p_\alpha \\
			&= L^{(r+d)}\left( \sum_{\vert \beta \vert= r+d}\sum_{\vert \alpha \vert = d}  {{r}\choose{\beta-\alpha}}p_\alpha \textbf{x}^\beta \right) \\
			& = L^{(r+d)}\left( \left(\sum_{i=1}^n x_i \right)^{r} p \right) \\
			& = L^{(r+d)}(p) \ge \underline{f}_{\text{LP}}^{(r+d)}.
			\end{align*} 
	\end{proof}

	%
	%
	%
	
	In the case of polynomial optimization our estimate \eqref{thmequ} becomes
	\[ f_{\min} - \underline{f}_{\text{LP}}^{(r+d)} \le \frac{d(d-1)}{2(r+d-1)-d(d-1)}(B(f)-f_{\min}) \]
	and applying the inequality 
	\[
	B(p)-p_{\min} \le {{2d-1}\choose{d}}d^d\left(p_{\max}-p_{\min}\right),
	\]
	shown in \cite[Theorem 2.2]{PTAS},  we find
	\[ f_{\min} - \underline{f}_{\text{LP}}^{(r+d)} \le \frac{d(d-1)}{2(r+d-1)-d(d-1)}{{2d-1}\choose{d}}d^d\left(f_{\max}-f_{\min}\right). \]
	This is essentially the same result as was obtained in \cite[Theorem 1.3]{PTAS}.

	\bibliographystyle{plain}       
	\bibliography{references2}   
	

	\end{document}